\documentclass[10pt]{amsart}

\usepackage{a4wide}
\usepackage{amsfonts,amssymb,amsmath,amscd,amstext}
\usepackage{mathrsfs}
\usepackage[colorlinks=true,linkcolor=blue,citecolor=blue]{hyperref}
\usepackage[utf8]{inputenc}
\usepackage{microtype}
\usepackage{graphicx}
\usepackage{changes}
\usepackage{comment}
\usepackage{mathtools}
\renewcommand{\leq}{\leqslant}
\renewcommand{\geq}{\geqslant}

\newcommand{\fr}{\penalty-20\null\hfill\(\blacksquare\)}

\newcommand{\nchi}{{\raise.3ex\hbox{\(\chi\)}}}
\newcommand{\N}{\mathbb{N}}
\newcommand{\R}{\mathbb{R}}

\renewcommand{\d}{{\mathrm d}}

\newcommand{\X}{{\rm X}}
\newcommand{\E}{{\rm E}}

\newcommand{\Lip}{{\rm Lip}}

\newcommand{\sfd}{{\sf d}}
\newcommand{\sfD}{{\sf D}}

\newcommand{\average}{{\mathchoice {\kern1ex\vcenter{\hrule height.4pt
				width 6pt
				depth0pt} \kern-9.7pt} {\kern1ex\vcenter{\hrule height.4pt width 4.3pt
				depth0pt}
			\kern-7pt} {} {} }}

\newtheorem{theorem}{Theorem}[section]
\newtheorem{proposition}[theorem]{Proposition}
\newtheorem{lemma}[theorem]{Lemma}
\newtheorem{corollary}[theorem]{Corollary}

\theoremstyle{definition}

\newtheorem{remark}[theorem]{Remark}

\newtheorem{definition}[theorem]{Definition} 

\theoremstyle{remark}

\numberwithin{equation}{section}


\begin{document}
\title[Variational problems concerning length distances in metric spaces]{Variational problems concerning \\ length distances in metric spaces}

\author[F. Essebei]{Fares Essebei}
\address[F. Essebei]{Dipartimento di Ingegneria Industriale e Scienze Matematiche, Universit\`a Politecnica delle Marche, Via Brecce Bianche, 60131 Ancona (Italy).}
\email{f.essebei@staff.univpm.it}

\author[E. Pasqualetto]{Enrico Pasqualetto}
\address[E. Pasqualetto]{Department of Mathematics and Statistics, P.O.\ Box 35 (MaD), FI-40014 University of Jyvaskyla, (Finland).}
\email{enrico.e.pasqualetto@jyu.fi}

\bibliographystyle{abbrv} 
\date{\today}
\maketitle
\begin{abstract}
Given a locally compact, complete metric space \(({\rm X},{\sf D})\) and an open set \(\Omega\subseteq{\rm X}\), we study the class of length distances \(\sf d\) on \(\Omega\)
that are bounded from above and below by fixed multiples of the ambient distance \(\sf D\). More precisely, we prove that the uniform convergence on compact sets of
distances in this class is equivalent to the \(\Gamma\)-convergence of several associated variational problems. Along the way, we fix some oversights appearing in the previous literature.
\end{abstract}
\section{Introduction}
\subsection*{General overview}
The convergence of several different variational functionals associated to length distances has been extensively studied in the literature.
Among others, it has been relevant e.g.\ to obtain homogenisation-type results for Riemannian and Finsler metrics \cite{AB,AV}, to study mass transport and optimisation problems \cite{BraButFra},
and to establish approximation results concerning Euclidean metrics and distances \cite{Smooth}. A key result is due to Buttazzo--De Pascale--Fragal\`{a} \cite{BDF}, where they
prove that the convergence of length distances on an open subset of the Euclidean space (that are controlled from above and below by fixed multiples of the Euclidean distance)
can be characterised in terms of the \(\Gamma\)-convergence of suitable length and energy functionals. Recently, we generalised this result to the sub-Riemannian setting \cite{esspasq}.
More specifically, we investigated the above-mentioned variational problems concerning distances in the case where the ambient space is a Carnot group endowed with a fixed sub-Riemannian
(or, more generally, sub-Finsler) distance.
\medskip

The aim of this paper is to extend the results discussed above to the much more general setting of (locally compact, complete) metric spaces. This shows that -- once the various objects
under consideration are defined in a suitable manner -- the principle behind the equivalence result is very robust and has little to do with the structure of the underlying ambient space.
Along the way, we also propose a solution to some issues we have found in the existing literature on the topic.
\medskip

The rest of the introduction is subdivided as follows: first we describe more in details our main results, then we compare them with the previous works in this field.
\subsection*{Statement of results}
Let \((\X,\sfD)\) be a locally compact, complete metric space, \(\Omega\subseteq\X\) an open set, and \(\alpha>1\) a given constant. Our main object of study is the class \(\mathcal D_\alpha(\Omega)\)
of all length distances on \(\Omega\) verifying \(\alpha^{-1}\sfD\leq\sfd\leq\alpha\sfD\), cf.\ with Definition \ref{classD}. We point out that the continuous extension \(\bar\sfd\) of \(\sfd\) to the
closure \(\overline{\Omega\times\Omega}\) is a geodesic distance on \(\bar\Omega\) verifying \(\alpha^{-1}\sfD\leq\bar\sfd\leq\alpha\sfD\), see Lemma \ref{lem:closure_geod}.
\medskip

We will consider three different functionals associated with any given distance \(\sfd\in\mathcal D_\alpha(\Omega)\):
\begin{itemize}
\item The length functional \(L_{\bar\sfd}\), which is defined on the space of Lipschitz curves from \([0,1]\) to the closure \(\bar\Omega\),
the latter being equipped with the topology of uniform convergence; see \eqref{lengthh}.
\item The `optimal-transport-type' functional \(J_\sfd\), which is defined as
\[
J_\sfd(\mu)\coloneqq\int\sfd(x,y)\,\d\mu(x,y)\quad\text{ for every finite Borel measure }\mu\geq 0\text{ on }\Omega\times\Omega;
\]
see \eqref{energy}. The domain of \(J_\sfd\) is equipped with the weak\(^*\) topology.
\item The `convex-minimisation-type' functional \(F_\sfd\), which is defined as
\[
F_\sfd(u)\coloneqq\left\{\begin{array}{ll}
0\\
+\infty
\end{array}\quad\begin{array}{ll}
\text{ if }u\text{ is }1\text{-Lipschitz with respect to }\sfd,\\
\text{ otherwise}
\end{array}\right.
\]
for every Lipschitz function \(u\colon\X\to\R\); see \eqref{F_d}. The domain of \(F_\sfd\) is equipped with the topology of uniform convergence on compact sets.
\end{itemize}
Our main result (i.e.\ Theorem \ref{thm:main_Gamma-conv}) states that if \((\sfd_n)_{n\in\N}\subseteq\mathcal D_\alpha(\Omega)\) and \(\sfd\in\mathcal D_\alpha(\Omega)\) are given, then
\[
\underset{\rm(A)}{\sfd_n\to\sfd}\qquad\Longleftrightarrow\qquad\underset{\rm(B)}{J_{\sfd_n}\overset{\Gamma}{\to}J_\sfd}\qquad
\Longleftrightarrow\qquad\underset{\rm(C)}{L_{\bar \sfd_n}\overset{\Gamma}{\to}L_{\bar\sfd}}\qquad\Longleftrightarrow\qquad\underset{\rm(D)}{F_{\sfd_n}\overset{\Gamma}{\to}F_\sfd},
\]
where the notation \(\overset{\Gamma}{\to}\) is used to indicate the \(\Gamma\)-convergence of the respective functionals. It is worth pointing out that in condition \(\rm(C)\)
it is necessary to require the \(\Gamma\)-convergence of the length functionals associated with the extended distances \(\bar\sfd_n\); we comment on this in the next paragraph.
\subsection*{Comparison with previous works}
Unfortunately, in the Euclidean version \cite[Theorem 3.1]{BDF} of the above equivalence result it is erroneously stated that \(\rm(A)\) is equivalent to the \(\Gamma\)-convergence
of the length functionals \(L_{\sfd_n}\) to \(L_\sfd\); here we are considering the original distances \(\sfd_n\) and \(\sfd\) defined on the open set \(\Omega\). This mistake
propagated to subsequent works, e.g.\ to our paper \cite{esspasq}; see Theorem 4.4 therein. This is the reason why we noticed this issue (regretfully, after the publication of
our paper). It holds true that \(\rm(A)\) implies that \(L_{\sfd_n}\) is \(\Gamma\)-converging to \(L_\sfd\), see Proposition
\ref{prop:converse_fails}, but the converse implication might fail. Indeed, in Lemma \ref{lem:counterex} we consider a square
\(Q\) in the plane and we construct a sequence of distances \((\sfd_n)_{n\in\N}\subseteq\mathcal D_2(Q)\) such that
\(L_{\sfd_n}\) \(\Gamma\)-converges to \(L_\sfd\) for some \(\sfd\in\mathcal D_2(Q)\), but for which \(\sfd_n\)
\emph{does not converge} to \(\sfd\); what happens is that \(\sfd_n\) converges to some smaller distance \(\tilde\sfd\),
which is not a length distance. In fact, a similar counterexample can be built on every open subset of \(\R^n\) with
\(n\geq 2\), see Remark \ref{rmk:ex_for_all_open}. This phenomenon is due to the fact that -- differently from what is
stated in the proofs of \cite[Theorem 3.1]{BDF} and in \cite[Corollary 4.3]{esspasq} -- the space \(\mathcal D_\alpha(\Omega)\)
\emph{is not compact} (in general). Technically speaking, the error in the proof of the compactness of \(\mathcal D_\alpha(\Omega)\)
is the following: one cannot deduce the convergence of the infima of the auxiliary functionals (e.g.\ \(L_{\sfd_n}^{x,y}\) in the
proof of \cite[Corollary 4.3]{esspasq}) from their \(\Gamma\)-convergence, since these functionals need not be equi-coercive.
As the equivalence between \(\rm(A)\) and \(\rm(C)\) illustrates, a possible way to fix the problem is to consider the length
functionals on the closure of \(\Omega\) instead.
\medskip

Since the proofs of (the analogues of) the implications \({\rm(B)}\Longrightarrow{\rm(A)}\) and
\({\rm(D)}\Longrightarrow{\rm(A)}\) were based on the compactness of \(\mathcal D_\alpha(\Omega)\),
new proof arguments are needed here. To the best of our knowledge, the error we discussed above
has not been previously observed. However, we mention that for instance Briani--Davini \cite{BriDa}
show that the class of geodesic Finsler distances on the closed set \(\bar\Omega\) is a compact metrisable space.
We also point out that the erratum \cite{BDFerr} is concerned with a different issue, which does not occur in the
present paper, thanks to the definition of \(F_\sfd\) we chose, which -- in the general metric setting -- seems
to be the only possible reasonable axiomatisation.
\bigskip

\noindent\textbf{Acknowledgements.}
The authors thank Tapio Rajala for several useful discussions on the topics of this paper.

\section{Preliminaries}
Let \((\E_1,\sfd_1)\) and \((\E_2,\sfd_2)\) be metric spaces. Then we denote by \(C(\E_1;\E_2)\) and \(\Lip(\E_1;\E_2)\) the spaces
of all continuous and Lipschitz maps, respectively, from \(\E_1\) to \(\E_2\). We endow both \(C(\E_1;\E_2)\) and
\(\Lip(\E_1;\E_2)\) with the topology of uniform convergence on compact sets. When \((\E_2,\sfd_2)\) is the real line \(\R\) together
with the Euclidean distance, we shorten \(C(\E_1;\R)\) and \(\Lip(\E_1;\R)\) as \(C(\E_1)\) and \(\Lip(\E_1)\), respectively.
Given any \(u\in\Lip(\E_1)\), we denote by \(\Lip_{\sfd_1}(u)\) its Lipschitz constant.
\begin{remark}\label{rmk:repr_dist}{\rm
Given any metric space \((\E,\sfd)\), it holds that
\[
\sfd(x,y)=\sup\big\{u(x)-u(y)\;\big|\;u\in\Lip(\E),\,\Lip_\sfd(u)\leq 1\big\}\quad\text{ for every }x,y\in\E.
\]
Indeed, the inequality \(\geq\) is trivial, while the inequality \(\leq\) follows by choosing \(u\coloneqq\sfd(\cdot,y)\).
\fr}\end{remark}
\subsection*{Some topological facts}
A topological space \(({\rm T},\tau)\) is said to be \emph{hemicompact} if there exists a sequence \((K_n)_{n\in\N}\) of compact subsets
of \({\rm T}\) having the following property: given any compact subset \(K\) of \({\rm T}\), there exists \(n\in\N\) such that \(K\subseteq K_n\).
In particular, since all singletons are compact, it holds that \({\rm T}=\bigcup_{n\in\N}K_n\).
\begin{lemma}\label{lem:about_hemicpt}
Every locally compact, connected metric space is hemicompact.
\end{lemma}
\begin{proof}
Let \((\E,\sfd)\) be a locally compact, connected metric space. Since metric spaces are paracompact
(see \cite{Stone}), we know from \cite[Appendix A]{Spivak}
that \((\E,\sfd)\) is \(\sigma\)-compact. Namely, there exists a sequence \((K_n)_{n\in\N}\) of compact subsets of \(\E\)
such that \(\E=\bigcup_{n\in\N}K_n\). Given any \(n\in\N\), we can find a finite family \(\mathcal F_n\) of compact subsets
of \(\E\) such that \(K_n\) is contained in the interior of \(\bigcup_{H\in\mathcal F_n}H\). It is then easy to check that
the set \(\mathcal C\) of all finite unions of elements of \(\bigcup_{n\in\N}\mathcal F_n\) is a countable family of
compact subsets of \(\E\) satisfying the following property: given any \(K\subseteq\E\) compact, there exists
\(H\in\mathcal C\) such that \(K\subseteq H\). This proves that \((\E,\sfd)\) is hemicompact, as required.
\end{proof}
\subsection*{The length functional \texorpdfstring{\(L_\sfd\)}{Ld}}
Given a metric space \((\E,\sfd)\) and a continuous curve \(\gamma\colon[a,b]\to\E\), we define the \emph{length} of \(\gamma\) as
\begin{equation}\label{lengthh}
L_\sfd(\gamma)\coloneqq\sup\bigg\{\sum_{i=1}^k\sfd(\gamma_{t_i},\gamma_{t_{i-1}})\;\bigg|\;k\in\N,\,a=t_0<t_1<\ldots<t_k=b\bigg\}.
\end{equation}
If \(L_\sfd(\gamma)<+\infty\), then we say that \(\gamma\) is \emph{rectifiable}. Notice that every Lipschitz curve is rectifiable.
\begin{remark}\label{rmk:L_d_lsc}{\rm
The length functional \(L_\sfd\colon C([0,1];\E)\to[0,+\infty]\) is lower semicontinuous. Indeed, it can expressed as the supremum
of the continuous functionals \(C([0,1];\E)\ni\gamma\mapsto\sum_{i=1}^k\sfd(\gamma_{t_i},\gamma_{t_{i-1}})\), where \((t_i)_{i=0}^k\)
varies in the family of all finite partitions of the interval \([0,1]\).
\fr}\end{remark}

A metric space \((\E,\sfd)\) is said to be a \emph{length space} provided it holds that
\[
\sfd(x,y)=\inf\Big\{L_\sfd(\gamma)\;\Big|\;\gamma\in\Lip([0,1];\E),\,\gamma_0=x,\,\gamma_1=y\Big\}\quad\text{ for every }x,y\in\E.
\]
We say that a rectifiable curve \(\gamma\colon[a,b]\to\E\) has \emph{constant-speed} provided it satisfies
\[
L_\sfd(\gamma|_{[t,s]})=\frac{L_\sfd(\gamma)}{|a-b|}|t-s|\quad\text{ for every }t,s\in[a,b]\text{ with }t<s.
\]
Each rectifiable curve \(\gamma\colon[a,b]\to\E\) admits a \emph{constant-speed reparameterisation}
\(\hat\gamma\colon[0,1]\to\E\), i.e.\ \(\hat\gamma\) is a rectifiable curve with
\((\hat\gamma_0,\hat\gamma_1)=(\gamma_a,\gamma_b)\), having constant-speed and the same image of \(\gamma\),
and satisfying \(L_\sfd(\hat\gamma)=L_\sfd(\gamma)\). A curve \(\gamma\colon[a,b]\to\E\) is a \emph{geodesic}
if it has constant-speed and
\[
L_\sfd(\gamma|_{[t,s]})=\sfd(\gamma_t,\gamma_s)\quad\text{ for every }t,s\in[a,b]\text{ with }t<s.
\]
We say that \((\E,\sfd)\) is a \emph{geodesic space} if for every \(x,y\in\E\)
there exists a geodesic \(\gamma\colon[0,1]\to\E\) such that \(\gamma_0=x\) and \(\gamma_1=y\). Every geodesic space is a length space.
The converse implication, which is in general false, holds for example when the metric space is locally compact and complete.
\subsection*{The class \texorpdfstring{\(\mathcal D_\alpha(\Omega)\)}{Dalpha(Omega)}}
Next we introduce the class of distances we are concerned with in this paper.
\begin{definition}\label{classD}
Let \((\X,\sfD)\) be a metric space and let \(\alpha>1\). Then we denote by \(\mathcal D_\alpha(\Omega,\sfD)\),
or by \(\mathcal D_\alpha(\Omega)\) for brevity, the set of all those distances \(\sfd\) on \(\Omega\) such that
\((\Omega,\sfd)\) is a length space and
\[
\frac{1}{\alpha}\sfD(x,y)\leq\sfd(x,y)\leq\alpha\sfD(x,y)\quad\text{ for every }x,y\in\Omega.
\]
We endow \(\mathcal D_\alpha(\Omega)\) with the topology of uniform convergence on compact subsets of \(\Omega\times\Omega\).
\end{definition}

The class \(\mathcal D_\alpha(\Omega)\) might be empty, for example if \(\Omega\) disconnected.  Given any \(\sfd\in\mathcal D_\alpha(\Omega)\), one has that
\(\Lip_\sfd([0,1];\Omega)=\Lip_\sfD([0,1];\Omega)\) and \(\alpha^{-1}L_\sfD(\gamma)\leq L_\sfd(\gamma)\leq\alpha L_\sfD(\gamma)\) for every \(\gamma\in\Lip([0,1];\Omega)\).
\begin{lemma}\label{lem:cor_Ascoli-Arzela}
Let \((\E,\sfD)\) be a hemicompact metric space and let \(\alpha>1\). Let \((\sfd_n)_{n\in\N}\) be given distances on \(\E\) such that \(\alpha^{-1}\sfD\leq\sfd_n\leq\alpha\sfD\)
for every \(n\in\N\). Then we can extract a subsequence \((n_i)_{i\in\N}\) such that \(\sfd_{n_i}\to\sfd\) uniformly on compact subsets of \(\E\times\E\), for some limit distance \(\sfd\) on \(\E\)
satisfying the inequalities \(\alpha^{-1}\sfD\leq\sfd\leq\alpha\sfD\).
\end{lemma}
\begin{proof}
Fix any increasing sequence \((K_j)_{j\in\N}\) of compact subsets of \(\E\) having the property that any compact subset of \(\E\)
is contained in \(K_j\) for some \(j\in\N\). Notice that for any \(n\in\N\) we can estimate
\[
\big|\sfd_n(x,y)-\sfd_n(\tilde x,\tilde y)\big|\leq\sfd_n(x,\tilde x)+\sfd_n(y,\tilde y)
\leq\alpha\big(\sfD(x,\tilde x)+\sfD(y,\tilde y)\big)\quad\text{ for every }x,\tilde x,y,\tilde y\in\E.
\]
This shows that each function \(\sfd_n\colon\E\times\E\to[0,+\infty)\) is Lipschitz if its domain is endowed with the
distance \((\sfD\times_1\sfD)\big((x,y),(\tilde x,\tilde y)\big)\coloneqq\sfD(x,\tilde x)+\sfD(y,\tilde y)\) and
that \(\Lip_{\sfD\times_1\sfD}(\sfd_n)\leq\alpha\). Therefore, thanks to the Arzel\`{a}--Ascoli theorem and a
diagonal argument, we can extract a subsequence \((n_i)_{i\in\N}\) such that \(\sfd_{n_i}\to\sfd\) uniformly on
each set \(K_j\times K_j\), for some function \(\sfd\colon\E\times\E\to[0,+\infty)\). By approximation, one can readily check
that \(\sfd\) is a pseudodistance on \(\E\) and that \(\alpha^{-1}\sfD\leq\sfd\leq\alpha\sfD\). The bound
\(\alpha^{-1}\sfD\leq\sfd\) implies that \(\sfd\) is a distance. Finally, given a compact subset \(H\) of \(\E\times\E\),
we have that \(H\subseteq\pi_1(H)\times\pi_2(H)\), where the maps \(\pi_1,\pi_2\colon\E\times\E\to\E\) are defined as
\(\pi_1(x,y)\coloneqq x\) and \(\pi_2(x,y)\coloneqq y\). Since \(\pi_1\) and \(\pi_2\) are continuous, we have
that \(\pi_1(H)\) and \(\pi_2(H)\) are compact, so that \(\pi_1(H),\pi_2(H)\subseteq K_j\) for some \(j\in\N\)
and accordingly \(H\subseteq K_j\times K_j\). This implies that \(\sfd_{n_i}\to\sfd\) uniformly on compact
subsets of \(\E\times\E\). Consequently, the proof of the statement is achieved.
\end{proof}
\begin{definition}
Let \((\X,\sfD)\) be a metric space. Let \(\Omega\subseteq\X\) be an open set. Let \(\sfd\) be a distance on \(\Omega\) such that \(\alpha^{-1}\sfD\leq\sfd\leq\alpha\sfD\)
on \(\Omega\times\Omega\) for some \(\alpha>1\). Then we denote by \(\bar\sfd\colon\bar\Omega\times\bar\Omega\to[0,+\infty)\) the unique continuous extension of the
function \(\sfd\) to \(\overline{\Omega\times\Omega}=\bar\Omega\times\bar\Omega\).
\end{definition}

It can be readily checked that \(\bar\sfd\) is a distance on \(\bar\Omega\) satisfying \(\alpha^{-1}\sfD\leq\bar\sfd\leq\alpha\sfD\) on \(\bar\Omega\times\bar\Omega\).
\begin{lemma}\label{lem:closure_geod}
Let \((\X,\sfD)\) be a locally compact, complete metric space and let \(\alpha>1\). Let \(\sfd\in\mathcal D_\alpha(\Omega)\) be a given distance.
Then the metric space \((\bar\Omega,\bar\sfd)\) is a geodesic space.
\end{lemma}
\begin{proof}
Let \(x,y\in\bar\Omega\) be given. Pick sequences \((x_n)_{n\in\N},(y_n)_{n\in\N}\subseteq\Omega\) such that \(x_n\to x\) and \(y_n\to y\). Since \((\Omega,\sfd)\) is a
length space, for any \(n\in\N\) we can find a curve \(\gamma^n\in\Lip([0,1];\Omega)\) such that
\[
L_\sfd(\gamma^n)\leq\sfd(x_n,y_n)+\frac{1}{n},\quad\gamma^n_0=x_n,\quad\gamma^n_1=y_n.
\]
We can assume without loss of generality that each curve \(\gamma^n\) has constant-speed with respect to the distance \(\sfd\), thus
the Lipschitz constant of \(\gamma^n\) is at most \(\sfd(x_n,y_n)+1\). Since \(\sfd(x_n,y_n)\to\bar\sfd(x,y)\) as \(n\to\infty\), we deduce that
\((\gamma^n)_{n\in\N}\) is an equi-Lipschitz family of curves, so that an application of the Arzel\`{a}--Ascoli theorem ensures that \(\gamma^{n_i}\to\gamma\)
uniformly for some subsequence \((n_i)_{i\in\N}\) and some curve \(\gamma\in\Lip([0,1];\bar\Omega)\). Notice that
\(\gamma_0=\lim_{i\to\infty}\gamma^{n_i}_0=\lim_{i\to\infty}x_{n_i}=x\) and similarly \(\gamma_1=y\). Then the lower semicontinuity of the length functional \(L_{\bar\sfd}\)
(recall Remark \ref{rmk:L_d_lsc}) implies that
\[
L_{\bar\sfd}(\gamma)\leq\liminf_{i\to\infty}L_{\bar\sfd}(\gamma^{n_i})=\liminf_{i\to\infty}L_\sfd(\gamma^{n_i})\leq\lim_{i\to\infty}\sfd(x_{n_i},y_{n_i})+\frac{1}{n_i}=\bar\sfd(x,y).
\]
Since also the converse inequality \(\bar\sfd(x,y)\leq L_{\bar\sfd}(\gamma)\) is verified, we conclude that \(L_{\bar\sfd}(\gamma)=\bar\sfd(x,y)\).
Hence, the constant-speed reparameterisation of \(\gamma\) is a geodesic, thus \((\bar\Omega,\bar\sfd)\) is a geodesic space.
\end{proof}
\begin{lemma}\label{lem:conv_on_closure}
Let \((\X,\sfD)\) be a locally compact, complete metric space and let \(\alpha>1\). Let \(\Omega\subseteq\X\) be an open set. Let \((\sfd_n)_{n\in\N}\) be
distances on \(\Omega\) with \(\alpha^{-1}\sfD\leq\sfd_n\leq\alpha\sfD\) on \(\Omega\times\Omega\) and \(\sfd_n\to\sfd\) uniformly on compact sets,
for some distance \(\sfd\) on \(\Omega\). Then \(\bar\sfd_n\to\bar\sfd\) uniformly on compact subsets of \(\bar\Omega\times\bar\Omega\).
\end{lemma}
\begin{proof}
First of all, notice that \(\bar\Omega\) is locally compact (as it is a closed subset of \(\X\)).
Moreover, since we are assuming that \(\mathcal D_\alpha(\Omega)\) is not empty, we know from Lemma \ref{lem:closure_geod}
that \(\bar\Omega\) is pathwise connected, thus in particular it is connected. Therefore, Lemma \ref{lem:about_hemicpt}
ensures that \(\bar\Omega\) is hemicompact, so that Lemma \ref{lem:cor_Ascoli-Arzela} tells that any subsequence of
\((\bar\sfd_n)_{n\in\N}\) admits a subsequence \((\bar\sfd_{n_i})_{i\in\N}\) such that \(\bar\sfd_{n_i}\to\tilde\sfd\)
uniformly on compact subsets of \(\bar\Omega\times\bar\Omega\), for some distance \(\tilde\sfd\) on \(\bar\Omega\).
Notice that \(\tilde\sfd|_{\Omega\times\Omega}=\sfd\), thus we conclude that \(\tilde\sfd=\bar\sfd\) and accordingly
\(\bar\sfd_n\to\bar\sfd\) uniformly on compact subsets of \(\bar\Omega\times\bar\Omega\).
\end{proof}
\subsection*{The functional \texorpdfstring{\(J_\sfd\)}{Jd}}
Let \((\X,\sfD)\) be a locally compact, complete metric space and let \(\alpha>1\). Let \(\Omega\subseteq\X\) be an open set.
We denote by \(\mathfrak M(\Omega\times\Omega)\) the Banach space of all finite signed Borel measures on \(\Omega\times\Omega\), endowed with the total variation norm. Then we define
the space \(\mathcal B(\Omega)\) as
\[
\mathcal B(\Omega)\coloneqq\big\{\mu\in\mathfrak M(\Omega\times\Omega)\;\big|\;\mu\geq 0\big\}.
\]
Since \(\Omega\times\Omega\) is a locally compact, Hausdorff topological space, we know that \(\mathfrak M(\Omega\times\Omega)\) is the dual Banach space of \(C_0(\Omega\times\Omega)\).
Recall that \(C_0(\Omega\times\Omega)\) is defined as the closure in \(C(\Omega\times\Omega)\) of the space \(C_c(\Omega\times\Omega)\) of compactly-supported, real-valued continuous functions
on \(\Omega\times\Omega\). We then endow the space \(\mathcal B(\Omega)\) with the topology induced by the weak\(^*\) topology of \(\mathfrak M(\Omega\times\Omega)\cong C_0(\Omega\times\Omega)^*\).
\medskip

Given any distance \(\sfd\in\mathscr D_\alpha(\Omega)\), we define the functional \(J_\sfd\colon\mathcal B(\Omega)\to[0,+\infty]\) as
\begin{equation}\label{energy}
J_\sfd(\mu)\coloneqq\int\sfd(x,y)\,\d\mu(x,y)\quad\text{ for every }\mu\in\mathcal B(\Omega).
\end{equation}
\subsection*{The functional \texorpdfstring{\(F_\sfd\)}{Fd}}
Let \((\X,\sfD)\) be a metric space and let \(\alpha>1\). Let \(\Omega\subseteq\X\) be an open set.
Given any distance \(\sfd\in\mathcal D_\alpha(\Omega)\), we define the functional \(F_\sfd\colon\Lip(\Omega)\to[0,+\infty]\) as
\begin{equation}\label{F_d}
F_\sfd(u)\coloneqq\left\{\begin{array}{ll}
0\\
+\infty
\end{array}\quad\begin{array}{ll}
\text{ if }\Lip_\sfd(u)\leq 1,\\
\text{ otherwise.}
\end{array}\right.
\end{equation}
\section{Main result}
This section is entirely devoted to the main result of the present paper. For an account of the theory
of \(\Gamma\)-convergence, which we will need in the next statement, we refer to the monograph \cite{dalmaso}.
\begin{theorem}\label{thm:main_Gamma-conv}
Let \((\X,\sfD)\) be a locally compact, complete metric space and let \(\alpha>1\). Let \(\Omega\subseteq\X\) be an open
set satisfying \(\mathcal D_\alpha(\Omega)\neq\varnothing\). Let \((\sfd_n)_{n\in\N}\subseteq\mathcal D_\alpha(\Omega)\)
and \(\sfd\in\mathcal D_\alpha(\Omega)\) be given distances. Then the following conditions are equivalent:
\begin{itemize}
\item[\(\rm (i)\)] \(\sfd_n\to\sfd\) in \(\mathcal D_\alpha(\Omega)\).
\item[\(\rm (ii)\)] \(J_{\sfd_n}\overset\Gamma\to J_\sfd\) in \(\mathcal B(\Omega)\).
\item[\(\rm (iii)\)] \(L_{\bar\sfd_n}\overset\Gamma\to L_{\bar\sfd}\) in \(\Lip([0,1];\bar\Omega)\).
\item[\(\rm (iv)\)] \(F_{\sfd_n}\overset\Gamma\to F_\sfd\) in \(\Lip(\Omega)\).
\end{itemize}
\end{theorem}
\begin{proof}
\ \\
\(\boldsymbol{{\rm (i)}\Longrightarrow{\rm (ii)}}\). The proof of this implication can be obtained by arguing exactly as for the
corresponding implication in \cite[Theorem 4.4]{esspasq}. Indeed, the argument therein is purely metric and does not rely
on the additional structure of the ambient space. Notice only that, in our framework, the existence of the cut-off functions
\((\eta_k)_k\) can be justified by exploiting the hemicompactness of \(\Omega\).
\medskip

\noindent\(\boldsymbol{{\rm (ii)}\Longrightarrow{\rm (i)}}\). Assuming \(J_{\sfd_n}\overset\Gamma\to J_\sfd\) in \(\mathcal B(\Omega)\),
we aim to prove that \(\sfd_n\to\sfd\) in \(\mathcal D_\alpha(\Omega)\).
Thanks to Lemma \ref{lem:cor_Ascoli-Arzela}, any subsequence of \((\sfd_n)_{n\in\N}\) admits a subsequence
\((\sfd_{n_i})_{i\in\N}\) such that \(\sfd_{n_i}\to\tilde\sfd\) uniformly on compact sets, for some distance \(\tilde\sfd\)
on \(\Omega\) with \(\alpha^{-1}\sfD\leq\tilde\sfd\leq\alpha\sfD\) on \(\Omega\times\Omega\). Notice that if we show that
\(\tilde\sfd=\sfd\), then we can conclude that \(\sfd_n\to\sfd\) in \(\mathcal D_\alpha(\Omega)\). Let \(\bar x,\bar y\in\Omega\) be fixed.
On the one hand,
\[
\sfd(\bar x,\bar y)=J_\sfd(\delta_{(\bar x,\bar y)})\leq\liminf_{n\to\infty}J_{\sfd_n}(\delta_{(\bar x,\bar y)})
\leq\lim_{i\to\infty}J_{\sfd_{n_i}}(\delta_{(\bar x,\bar y)})=\lim_{i\to\infty}\sfd_{n_i}(\bar x,\bar y)=\tilde\sfd(\bar x,\bar y)
\]
by the \(\Gamma\)-liminf inequality, where \(\delta_{(\bar x,\bar y)}\) stands for the Dirac delta measure at \((\bar x,\bar y)\).
On the other hand, by the \(\Gamma\)-limsup inequality we can find a sequence \((\mu_n)_{n\in\N}\subseteq\mathcal B(\Omega)\) that weakly\(^*\)
converges to \(\delta_{(\bar x,\bar y)}\) and satisfies \(J_\sfd(\delta_{(\bar x,\bar y)})=\lim_{n\to\infty}J_{\sfd_n}(\mu_n)\).
Given \(\varepsilon>0\), choose \(i_0\in\N\) such that
\[
\big|\tilde\sfd(\bar x,\bar y)-\sfd_{n_i}(\bar x,\bar y)\big|\leq\frac{\varepsilon}{2}\quad\text{ for every  }i\geq i_0.
\]
By the local compactness of \((\X,\sfD)\), there exists a radius \(r\in(0,\frac{\varepsilon}{4\alpha})\) such that
\[
K\coloneqq\big\{(x,y)\in\X\times\X\;\big|\;\sfD(x,\bar x),\sfD(y,\bar y)\leq r\big\}
\]
is compact and contained in \(\Omega\times\Omega\).
Fix any \(\eta\colon\Omega\times\Omega\to[0,1]\) continuous with \(\eta=0\) in \((\Omega\times\Omega)\setminus K\) and \(\eta(\bar x,\bar y)=1\).
The cut-off function \(\eta\) belongs to \(C_c(\Omega\times\Omega)\), thus \(\lambda_i\coloneqq\int\eta\,\d\mu_{n_i}\to\int\eta\,\d\delta_{(\bar x,\bar y)}=1\).
Now define the Borel probability measures \(\nu_i\) on \(\Omega\times\Omega\) as \(\nu_i\coloneqq\lambda_i^{-1}\eta\mu_{n_i}\). For \(i\geq i_0\) we can estimate
\[\begin{split}
\bigg|\int\sfd_{n_i}(x,y)\,\d\nu_i(x,y)-\tilde\sfd(\bar x,\bar y)\bigg|&\leq\bigg|\int\sfd_{n_i}(x,y)\,\d\nu_i(x,y)-\sfd_{n_i}(\bar x,\bar y)\bigg|+\big|\sfd_{n_i}(\bar x,\bar y)-\tilde\sfd(\bar x,\bar y)\big|\\
&\leq\int\big|\sfd_{n_i}(x,y)-\sfd_{n_i}(\bar x,\bar y)\big|\,\d\nu_i(x,y)+\frac{\varepsilon}{2}\\
&\leq\int\big(\sfd_{n_i}(x,\bar x)+\sfd_{n_i}(y,\bar y)\big)\,\d\nu_i(x,y)+\frac{\varepsilon}{2}\\
&\leq\alpha\int_K\big(\sfD(x,\bar x)+\sfD(y,\bar y)\big)\,\d\nu_i(x,y)+\frac{\varepsilon}{2}\\
&\leq 2\alpha r+\frac{\varepsilon}{2}<\varepsilon.
\end{split}\]
Hence, recalling that \(\eta\leq 1\) and thus \(\nu_i\leq\lambda_i^{-1}\mu_{n_i}\), we finally conclude that
\[
\tilde\sfd(\bar x,\bar y)\leq\varepsilon+\liminf_{i\to\infty}\int\sfd_{n_i}(x,y)\,\d\nu_i(x,y)\leq\varepsilon+\lim_{i\to\infty}\frac{1}{\lambda_i}J_{\sfd_{n_i}}(\mu_{n_i})
=\varepsilon+J_\sfd(\delta_{(\bar x,\bar y)})=\varepsilon+\sfd(\bar x,\bar y).
\]
Letting \(\varepsilon\to 0\), we deduce that \(\tilde\sfd(\bar x,\bar y)\leq\sfd(\bar x,\bar y)\) and thus \(\tilde\sfd=\sfd\). Therefore, item (i) is proved.
\medskip

\noindent\(\boldsymbol{{\rm (i)}\Longrightarrow{\rm (iii)}}\). Assuming \(\sfd_n\to\sfd\) in \(\mathcal D_\alpha(\Omega)\), we aim to prove that
\(L_{\bar\sfd_n}\overset\Gamma\to L_{\bar\sfd}\) in \(\Lip([0,1];\bar\Omega)\). In order to check the \(\Gamma\)-liminf inequality,
fix an arbitrary converging sequence \(\gamma^n\to\gamma\) in \(\Lip([0,1];\bar\Omega)\). Since the image of \(\gamma\) is a compact subset
of \(\bar\Omega\) and the sequence \((\gamma^n)_{n\in\N}\) converges to \(\gamma\) uniformly, we can find \(\bar n\in\N\) and a compact set
\(K\subseteq\bar\Omega\) such that \(\gamma^n_t\in K\) for every \(n\geq\bar n\) and \(t\in[0,1]\). We know from Lemma \ref{lem:conv_on_closure}
that \(\bar\sfd_n\to\bar\sfd\) uniformly on \(K\times K\). Given a partition \(0=t_0<t_1<\ldots<t_k=1\) of \([0,1]\), we thus have that
\(\bar\sfd_n(\gamma^n_{t_i},\gamma^n_{t_{i-1}})\to\bar\sfd(\gamma_{t_i},\gamma_{t_{i-1}})\) as \(n\to\infty\) for every \(i=1,\ldots,k\), so that
\[
\sum_{i=1}^k\bar\sfd(\gamma_{t_i},\gamma_{t_{i-1}})=\lim_{n\to\infty}\sum_{i=1}^k\bar\sfd_n(\gamma^n_{t_i},\gamma^n_{t_{i-1}})\leq\liminf_{n\to\infty}L_{\bar\sfd_n}(\gamma^n).
\]
By the arbitrariness of \(0=t_0<t_1<\ldots<t_k=1\), we deduce that
\(L_{\bar\sfd}(\gamma)\leq\liminf_{n\to\infty}L_{\bar\sfd_n}(\gamma^n)\).

Let us pass to the verification of the \(\Gamma\)-limsup inequality. Fix any curve \(\gamma\in\Lip([0,1];\bar\Omega)\).
Since the image \(K\) of \(\gamma\) is a compact subset of \(\bar\Omega\), we have that \(\sup_{K\times K}|\bar\sfd_n-\bar\sfd|\to 0\)
as \(n\to\infty\). Pick a sequence \((r_n)_{n\in\N}\subseteq\N\) such that \(r_n\to\infty\)
and 
\[
r_n\sup_{K\times K}|\bar\sfd_n-\bar\sfd|\to 0 \quad \mbox{as } n\to\infty.
\]
Given that \((\bar\Omega,\bar\sfd_n)\) is a geodesic space
by Lemma \ref{lem:closure_geod}, for any \(n\in\N\) we can find \(\gamma^n\in\Lip([0,1];\bar\Omega)\) such that
\({\gamma^n}|_{[(i-1)/r_n,i/r_n]}\) is a geodesic curve in \((\bar\Omega,\bar\sfd_n)\) and
\((\gamma^n_{(i-1)/r_n},\gamma^n_{i/r_n})=(\gamma_{(i-1)/r_n},\gamma_{i/r_n})\) for every \(i=1,\ldots,r_n\).
In particular, \(\gamma^n\) is a Lipschitz curve from \([0,1]\) to \((\bar\Omega,\sfD)\) whose Lipschitz constant does not exceed \(\alpha L_\sfD(\gamma)\).
Moreover, letting \(\lambda\geq 0\) be the Lipschitz constant of \(\gamma\) as a curve from \([0,1]\) to \((\bar\Omega,\sfD)\), we have
\(\sup_{t\in[0,1]}\sfD(\gamma^n_t,\gamma_t)\leq\frac{\lambda+\alpha L_\sfD(\gamma)}{r_n}\), thus \(\gamma^n\to\gamma\) uniformly. Finally,
\[\begin{split}
L_{\bar\sfd}(\gamma)&\geq\sum_{i=1}^{r_n}\bar\sfd(\gamma_{(i-1)/r_n},\gamma_{i/r_n})\geq\sum_{i=1}^{r_n}\bar\sfd_n(\gamma^n_{(i-1)/r_n},\gamma^n_{i/r_n})-r_n\sup_{K\times K}|\bar\sfd_n-\bar\sfd|\\
&=\sum_{i=1}^{r_n}L_{\bar\sfd_n}(\gamma^n|_{[(i-1)/r_n,i/r_n]})-r_n\sup_{K\times K}|\bar\sfd_n-\bar\sfd|=L_{\bar\sfd_n}(\gamma^n)-r_n\sup_{K\times K}|\bar\sfd_n-\bar\sfd|,
\end{split}\]
whence \(L_{\bar\sfd}(\gamma)\geq\limsup_{n\to\infty}L_{\bar\sfd_n}(\gamma^n)\) follows by letting \(n\to\infty\). Therefore, item (iii) is proved.
\medskip

\noindent\(\boldsymbol{{\rm (iii)}\Longrightarrow{\rm (i)}}\). Assuming \(L_{\bar\sfd_n}\overset\Gamma\to L_{\bar\sfd}\)
in \(\Lip([0,1];\bar\Omega)\), we aim to prove that \(\sfd_n\to\sfd\) in \(\mathcal D_\alpha(\Omega)\).
As before, it is sufficient to show that if \(\sfd_n\to\tilde\sfd\) uniformly on compact sets for some distance \(\tilde\sfd\)
on \(\Omega\), then \(\tilde\sfd=\sfd\). Let \(x,y\in\Omega\) be fixed. On the one hand, Lemma \ref{lem:closure_geod} ensures
that \((\bar\Omega,\bar\sfd)\) is a geodesic space, thus there exists a geodesic \(\gamma\colon[0,1]\to(\bar\Omega,\bar\sfd)\) with
\(\gamma_0=x\) and \(\gamma_1=y\). Chosen a sequence \((\gamma^n)_{n\in\N}\subseteq\Lip([0,1];\bar\Omega)\) such that \(\gamma^n\to\gamma\)
and \(L_{\bar\sfd_n}(\gamma^n)\to L_{\bar\sfd}(\gamma)\), one has \((\gamma^n_0,\gamma^n_1)\to(x,y)\) and
\[
\tilde\sfd(x,y)=\lim_{n\to\infty}\bar\sfd_n(\gamma^n_0,\gamma^n_1)\leq\lim_{n\to\infty}L_{\bar\sfd_n}(\gamma^n)=L_{\bar\sfd}(\gamma)=\bar\sfd(x,y)=\sfd(x,y).
\]
On the other hand, since each \((\bar\Omega,\bar\sfd_n)\) is a geodesic space again by Lemma \ref{lem:closure_geod}, for any \(n\in\N\) we can find
a geodesic \(\gamma^n\colon[0,1]\to(\bar\Omega,\bar\sfd_n)\) such that \(\gamma^n_0=x\) and \(\gamma^n_1=y\). Notice that the Lipschitz constant of
\(\gamma^n\) with respect to the distance \(\sfD\) cannot exceed \(\alpha^2\sfD(x,y)\), thus accordingly an application of the Arzel\`{a}--Ascoli theorem
provides us with a subsequence \((n_i)_{i\in\N}\subseteq\N\) and a curve \(\gamma\in\Lip([0,1];\bar\Omega)\) such that \(\gamma^{n_i}\to\gamma\) uniformly.
Notice that \(\gamma_0=x\) and \(\gamma_1=y\). Therefore, we have that
\[
\sfd(x,y)\leq L_{\bar\sfd}(\gamma)\leq\liminf_{i\to\infty}L_{\bar\sfd_{n_i}}(\gamma^{n_i})=\lim_{i\to\infty}\sfd_{n_i}(x,y)=\tilde\sfd(x,y)
\]
thanks to the \(\Gamma\)-liminf inequality. All in all, the identity \(\tilde\sfd=\sfd\) is proved, whence item (i) follows.
\medskip

\noindent\(\boldsymbol{{\rm (i)}\Longrightarrow{\rm (iv)}}\). Assuming \(\sfd_n\to\sfd\) in \(\mathcal D_\alpha(\Omega)\),
we aim to prove that \(F_{\sfd_n}\overset\Gamma\to F_\sfd\) in \(\Lip(\Omega)\). In order to check the \(\Gamma\)-liminf
inequality, fix a converging sequence \(u_n\to u\) in \(\Lip(\Omega)\). If \(F_\sfd(u)=0\), then there is nothing
to prove. If \(F_\sfd(u)=+\infty\), then we can find \(x,y\in\Omega\) such that \(u(x)-u(y)>\sfd(x,y)\). It follows
that there exists \(\bar n\in\N\) such that \(u_n(x)-u_n(y)>\sfd_n(x,y)\) for every \(n\geq\bar n\), thus in particular
\(\Lip_{\sfd_n}(u_n)>1\) for every \(n\geq\bar n\) and accordingly \(F_\sfd(u)=+\infty=\liminf_{n\to\infty}F_{\sfd_n}(u_n)\).

Let us pass to the verification of the \(\Gamma\)-limsup inequality. Fix any \(u\in\Lip(\Omega)\).
If \(F_\sfd(u)=+\infty\), then there is nothing to prove (since the sequence constantly equal to \(u\) is a recovery sequence).
Then we focus on the case where \(F_\sfd(u)=0\). Since \(\Omega\) is hemicompact, we can find an increasing sequence
\((K_i)_{i\in\N}\) of compact subsets of \(\Omega\) having the property that every compact subset of \(\Omega\) is contained
in \(K_i\) for some \(i\in\N\). For any \(i,n\in\N\), we define \(u_{n,i}\colon\Omega\to\R\) via inf-convolution as
\[
u_{n,i}(x)\coloneqq\inf_{y\in K_i}\big(u(y)+\sfd_n(x,y)\big)\quad\text{ for every }x\in\Omega.
\]
Since \(\Lip_{\sfd_n}(u(y)+\sfd_n(\cdot,y))\leq 1\) for every \(y\in K_i\), we deduce that \(u_{n,i}\in\Lip(\Omega)\)
and \(\Lip_{\sfd_n}(u_{n,i})\leq 1\), so that \(F_{\sfd_n}(u_{n,i})=0\). Given any \(x\in K_i\), we have that
\(u_{n,i}(x)=u(y_{n,i}(x))+\sfd_n(x,y_{n,i}(x))\) for some point \(y_{n,i}(x)\in K_i\), thus we can estimate
\[\begin{split}
u(x)&\geq u_{n,i}(x)=u(y_{n,i}(x))+\sfd_n(x,y_{n,i}(x))\geq u(x)-\sfd(x,y_{n,i}(x))+\sfd_n(x,y_{n,i}(x))\\
&\geq u(x)-\sup_{K_i\times K_i}|\sfd_n-\sfd|.
\end{split}\]
Given that for any \(i\in\N\) it holds that \(\sup_{K_i\times K_i}|\sfd_n-\sfd|\to 0\) as \(n\to\infty\),
we can extract a subsequence \((n_i)_{i\in\N}\) such that \(\sup_{K_i}|u-u_{n,i}|\leq\frac{1}{i}\) for
every \(i\in\N\) and \(n\geq n_i\). Define \((u_n)_{n\in\N}\) as
\[
u_n\coloneqq\left\{\begin{array}{ll}
u\\
u_{n,i}
\end{array}\quad\begin{array}{ll}
\text{ if }n<n_1,\\
\text{ if }n_i\leq n<n_{i+1}\text{ for some }i\in\N.
\end{array}\right.
\]
Given any \(K\subseteq\Omega\) compact and \(\varepsilon>0\), we can find \(j\in\N\) such that \(K\subseteq K_i\) and
\(\frac{1}{i}\leq\varepsilon\) hold for every \(i\geq j\), thus \(\sup_K|u-u_n|\leq\varepsilon\) for every \(n\geq n_j\).
Then \(u_n\to u\) uniformly on compact sets. Since \(F_\sfd(u)=0=\limsup_{n\to\infty}F_{\sfd_n}(u_n)\),
we showed that \((u_n)_{n\in\N}\) is a recovery sequence for \(u\).
\medskip

\noindent\(\boldsymbol{{\rm (iv)}\Longrightarrow{\rm (i)}}\). Assuming \(F_{\sfd_n}\overset\Gamma\to F_\sfd\) in \(\Lip(\Omega)\),
we aim to prove that \(\sfd_n\to\sfd\) in \(\mathcal D_\alpha(\Omega)\). As before, it suffices to show that if
\(\sfd_n\to\tilde\sfd\) uniformly on compact sets for some distance \(\tilde\sfd\) on \(\Omega\), then \(\tilde\sfd=\sfd\).
Then let \(x,y\in\Omega\) be fixed. On the one hand, the \(\Gamma\)-limsup inequality tells that for any \(u\in\Lip(\Omega)\)
with \(\Lip_\sfd(u)\leq 1\) we can find \((u_n)_{n\in\N}\subseteq\Lip(\Omega)\) converging to \(u\) with
\(\Lip_{\sfd_n}(u_n)\leq 1\), thus
\[
u(x)-u(y)=\lim_{n\to\infty}\big(u_n(x)-u_n(y)\big)\leq\lim_{n\to\infty}\sfd_n(x,y)=\tilde\sfd(x,y),
\]
whence (recalling Remark \ref{rmk:repr_dist}) it follows that
\[
\sfd(x,y)=\sup\Big\{u(x)-u(y)\;\Big|\;u\in\Lip(\Omega),\,\Lip_\sfd(u)\leq 1\Big\}\leq\tilde\sfd(x,y).
\]
On the other hand, we have \(\tilde\sfd(x,\cdot),\sfd_n(x,\cdot)\in\Lip(\Omega)\) and \(\Lip_{\sfd_n}(\sfd_n(x,\cdot))\leq 1\)
for every \(n\in\N\), thus the \(\Gamma\)-liminf inequality ensures that \(F_\sfd(\tilde\sfd(x,\cdot))\leq\liminf_{n\to\infty}F_{\sfd_n}(\sfd_n(x,\cdot))=0\),
whence it follows that \(\Lip_\sfd(\tilde\sfd(x,\cdot))\leq 1\) and thus \(\tilde\sfd(x,y)\leq\sfd(x,y)\). All in all, we have shown that \(\tilde\sfd=\sfd\).
\end{proof}
\section{Some counterexamples}
We denote by \(\sfd_{\rm Eucl}\) the Euclidean distance on \(\R^n\), i.e.\ \(\sfd_{\rm Eucl}(x,y)\coloneqq|x-y|\) for every \(x,y\in\R^n\).
\begin{lemma}\label{lem:counterex}
Consider the open square \(\Omega\coloneqq(0,1)^2\subseteq\R^2\) and \(\sfd\coloneqq 2\sfd_{\rm Eucl}|_{\Omega\times\Omega}\in\mathcal D_2(\Omega,\sfd_{\rm Eucl})\).
Then there exist an increasing sequence \((\sfd_n)_{n\in\N}\subseteq\mathcal D_2(\Omega,\sfd_{\rm Eucl})\)
and a distance \(\tilde\sfd\) on \(\Omega\) with \(\tilde\sfd\neq\sfd\) such that \(L_{\sfd_n}\overset\Gamma\to L_\sfd\)
in \(\Lip([0,1];\Omega)\) and \(\sfd_n\to\tilde\sfd\) uniformly on compact subsets of \(\Omega\times\Omega\).
\end{lemma}
\begin{proof}
Given any \(n\in\N\), fix a smooth function \(\phi_n\colon\Omega\to[1,2]\) satisfying
\[
\phi_n=2\;\text{ on }(0,1)\times[2^{-n},1),\quad\phi_n=1\;\text{ on }(0,1)\times(0,2^{-(n+1)}].
\]
We then define the smooth Riemannian metric \(\varphi_n\colon\Omega\times\R^2\to[0,+\infty)\)
as \(\varphi_n(x,v)\coloneqq\phi_n(x)|v|\) for every \(x\in\Omega\) and \(v\in\R^2\).
We denote by \(\sfd_n\colon\Omega\times\Omega\to[0,+\infty)\) the distance induced by \(\varphi_n\), i.e.
\[
\sfd_n(x,y)\coloneqq\inf\bigg\{\int_0^1\varphi_n(\gamma_t,\dot\gamma_t)\,\d t\;\bigg|\;
\gamma\in{\rm Lip}([0,1];\Omega),\,\gamma_0=x,\,\gamma_1=y\bigg\}\quad\text{ for every }x,y\in\Omega.
\]
Given that \(1\leq\phi_n\leq 2\), we have that \(\sfd_{\rm Eucl}\leq\sfd_n\leq\sfd\) in \(\Omega\times\Omega\).
Since \(\phi_n\) is continuous, we also know from \cite[Theorem 2.5]{DCP2} that
\(L_{\sfd_n}(\gamma)=\int_0^1\varphi_n(\gamma_t,\dot\gamma_t)\,\d t\) for every \(\gamma\in\Lip([0,1];\Omega)\) and thus
\((\Omega,\sfd_n)\) is a length space. In particular, it holds that \(\sfd_n\in\mathcal D_2(\Omega)\).
Since \((\sfd_n)_{n\in\N}\) is non-decreasing by construction, we have that the limit
\(\tilde\sfd(x,y)\coloneqq\lim_{n\to\infty}\sfd_n(x,y)\) exists for every \(x,y\in\Omega\). It can be readily
checked that \(\tilde\sfd\) is a distance satisfying \(\sfd_{\rm Eucl}\leq\tilde\sfd\leq\sfd\) in \(\Omega\times\Omega\).
In particular, \(\tilde\sfd\) is a continuous function and thus \(\sfd_n\to\tilde\sfd\) uniformly on compact subsets
of \(\Omega\times\Omega\). We claim that
\[
L_{\sfd_n}\overset\Gamma\to L_\sfd\quad\text{ in }\Lip([0,1];\Omega).
\]
In order to prove it, fix any converging sequence \(\gamma^n\to\gamma\) in \(\Lip([0,1];\Omega)\).
The image of \(\gamma\), which is a compact subset of \(\Omega\), is contained in \((0,1)\times(2\delta,1)\)
for some \(\delta\in\big(0,\frac{1}{2}\big)\). Therefore, we can find \(\bar n\in\N\) such that
\(\sup_{t\in[0,1]}|\gamma^n_t-\gamma_t|\leq\delta\) and \(2^{-n}\leq\delta\) holds for every \(n\geq\bar n\).
In particular, the image of \(\gamma^n\) is contained in \(\{\phi_n=2\}\) for every \(n\geq\bar n\),
so that \(L_{\sfd_n}(\gamma^n)=L_{2\sfd_{\rm Eucl}}(\gamma^n)=L_\sfd(\gamma^n)\) for every \(n\geq\bar n\).
Recalling that \(L_\sfd\) is lower semicontinuous (see Remark \ref{rmk:L_d_lsc}), we deduce that
\[
L_\sfd(\gamma)\leq\liminf_{n\to\infty}L_\sfd(\gamma^n)=\liminf_{n\to\infty}L_{\sfd_n}(\gamma^n),
\]
thus proving the \(\Gamma\)-liminf inequality. To prove the \(\Gamma\)-limsup inequality,
let \(\gamma\in\Lip([0,1];\Omega)\) be given. Arguing as before, we can find \(\bar n\in\N\)
such that \(L_{\sfd_n}(\gamma)=L_\sfd(\gamma)\) for every \(n\geq\bar n\), thus in particular
\(L_\sfd(\gamma)=\lim_{n\to\infty}L_{\sfd_n}(\gamma)\), which shows that the sequence constantly
equal to \(\gamma\) is a recovery sequence. All in all, we proved that \(L_{\sfd_n}\overset\Gamma\to L_\sfd\)
in \(\Lip([0,1];\Omega)\), as we claimed above.

In order to achieve the statement, it remains to check that \(\tilde\sfd\neq\sfd\). We denote
\(a\coloneqq\big(\frac{1}{16},\frac{1}{8}\big)\in\Omega\) and \(b\coloneqq\big(\frac{15}{16},\frac{1}{8}\big)\in\Omega\).
Given any \(n\in\N\) with \(n\geq 2\), we fix a curve \(\sigma^n\in\Lip([0,1];\Omega)\) such that
\[\begin{split}
\sigma^n|_{[0,\frac{1}{3}]}&\quad\text{ is a parameterisation of the interval }
\textstyle\big[a,\big(\frac{1}{16},\frac{1}{2^{n+1}}\big)\big],\\
\sigma^n|_{[\frac{1}{3},\frac{2}{3}]}&\quad\text{ is a parameterisation of the interval }
\textstyle\big[\big(\frac{1}{16},\frac{1}{2^{n+1}}\big),\big(\frac{15}{16},\frac{1}{2^{n+1}}\big)\big],\\
\sigma^n|_{[\frac{2}{3},1]}&\quad\text{ is a parameterisation of the interval }
\textstyle\big[\big(\frac{15}{16},\frac{1}{2^{n+1}}\big),b\big].
\end{split}\]
Using the fact that \(\phi_n=1\) on the image of \(\sigma^n|_{[\frac{1}{3},\frac{2}{3}]}\), and employing just the upper
bound \(\phi_n\leq 2\) on the images of \(\sigma^n|_{[0,\frac{1}{3}]}\) and \(\sigma^n|_{[\frac{2}{3},1]}\), for every
\(n\geq 2\) we can estimate
\[\begin{split}
\sfd_n(a,b)&\leq L_{\sfd_n}(\sigma^n)=L_{\sfd_n}(\sigma^n|_{[0,\frac{1}{3}]})
+L_{\sfd_n}(\sigma^n|_{[\frac{1}{3},\frac{2}{3}]})+L_{\sfd_n}(\sigma^n|_{[\frac{2}{3},1]})\\
&\leq 2\bigg(\frac{1}{8}-\frac{1}{2^{n+1}}\bigg)+\frac{7}{8}+2\bigg(\frac{1}{8}-\frac{1}{2^{n+1}}\bigg)<\frac{11}{8}.
\end{split}\]
It follows that \(\tilde\sfd(a,b)=\lim_{n\to\infty}\sfd_n(a,b)\leq\frac{11}{8}<\frac{7}{4}=\sfd(a,b)\),
which shows that \(\tilde\sfd\neq\sfd\).
\end{proof}
\begin{corollary}
There exists an open set \(\Omega\subseteq\R^2\) such that the space \(\mathcal D_2(\Omega)\neq\varnothing\) is not compact.
\end{corollary}
\begin{proof}
Let \(\Omega\), \((\sfd_n)_{n\in\N}\), and \(\tilde\sfd\) be as in Lemma \ref{lem:counterex}. We claim that \((\sfd_n)_{n\in\N}\) does not admit
any converging subsequence in \(\mathcal D_2(\Omega,\sfd_{\rm Eucl})\), whence the non-compactness of \(\mathcal D_2(\Omega,\sfd_{\rm Eucl})\) follows.
We argue by contradiction: suppose that \(\sfd_{n_i}\to\hat\sfd\) in \(\mathcal D_2(\Omega,\sfd_{\rm Eucl})\), for some subsequence \((n_i)_{i\in\N}\)
and some limit distance \(\hat\sfd\in\mathcal D_2(\Omega,\sfd_{\rm Eucl})\). Given that \(\sfd_{n_i}\to\hat\sfd\) uniformly on compact subsets of
\(\Omega\times\Omega\), we conclude that \(\tilde\sfd=\hat\sfd\). This leads to a contradiction, since we have that \(\tilde\sfd\notin\mathcal D_2(\Omega,\sfd_{\rm Eucl})\).
\end{proof}
\begin{remark}\label{rmk:ex_for_all_open}{\rm
More generally, by suitably adapting the construction in Lemma \ref{lem:counterex} one can prove that \(\mathcal D_\alpha(\Omega)\)
is not compact whenever \(n\geq 2\), \(\Omega\subseteq\R^n\) is open, \(\alpha>1\), and \(\mathcal D_\alpha(\Omega)\neq\varnothing\).
\fr}\end{remark}
\begin{proposition}\label{prop:converse_fails}
Let \((\X,\sfD)\) be a locally compact, complete metric space and let \(\alpha>1\).
Let \(\Omega\subseteq\X\) be an open set such that \(\mathcal D_\alpha(\Omega)\neq\varnothing\).
Then the implication
\[
\sfd_n\to\sfd\;\text{ in }\mathcal D_\alpha(\Omega)\quad\Longrightarrow\quad
L_{\sfd_n}\overset\Gamma\to L_\sfd\;\text{ in }\Lip([0,1];\Omega)
\]
holds, while there exist examples where the converse implication fails.
\end{proposition}
\begin{proof}
Suppose \(\sfd_n\to\sfd\) in \(\mathcal D_\alpha(\Omega)\). We thus know from the implication \({\rm(i)}\Longrightarrow{\rm(iii)}\) of Theorem \ref{thm:main_Gamma-conv}
that \(L_{\bar\sfd_n}\overset\Gamma\to L_{\bar\sfd}\) in \(\Lip([0,1];\bar\Omega)\). The \(\Gamma\)-convergence \(L_{\sfd_n}\overset\Gamma\to L_\sfd\) in \(\Lip([0,1];\Omega)\)
then follows by just observing that if \(\gamma^n\to\gamma\) in \(\Lip([0,1];\bar\Omega)\) and \(\gamma\in\Lip([0,1];\Omega)\), then \((\gamma^n)_{n\geq\bar n}\subseteq\Lip([0,1];\Omega)\)
for some \(\bar n\in\N\). Conversely, let us now consider \(\Omega\), \((\sfd_n)_{n\in\N}\), \(\sfd\), and \(\tilde\sfd\) as in Lemma \ref{lem:counterex}. Then we have that
\(L_{\sfd_n}\overset\Gamma\to L_\sfd\) in \(\Lip([0,1];\Omega)\), but \(\sfd_n\) cannot converge to \(\sfd\) in \(\mathcal D_2(\Omega,\sfd_{\rm Eucl})\) because \(\sfd\neq\tilde\sfd\).
\end{proof}

\def\cprime{$'$} \def\cprime{$'$}


\begin{thebibliography}{15}

\bibitem{AB}
{\sc E.~Acerbi and G.~Buttazzo}, {\em On the limits of periodic {R}iemannian
  metrics}, J. Anal. Math., 43 (1984), pp.~183--201.

\bibitem{AV}
{\sc M.~Amar and E.~Vitali}, {\em Homogenization of Periodic {F}insler
  metrics}, J. Convex Anal., 5 (1998), pp.~171--186.

\bibitem{BraButFra}
{\sc A.~Braides, G.~Buttazzo, and I.~Fragal\`{a} },
  {\em Riemannian approximation of {F}insler metrics }, Asymptotic Anal., 31 no.2 (2022), pp.~177--187.
  
\bibitem{BriDa}
{\sc A.~Briani and A.~Davini}, {\em Monge solutions for discontinuous {H}amiltonians}, ESAIM: Control, Optimisation and Calculus of Variations, 11
  (2005), pp.~229--251.  

\bibitem{BDF}
{\sc G.~Buttazzo, L.~De~Pascale, and I.~Fragal\`{a}}, {\em Topological
  equivalence of some variational problems involving distances}, Discrete
  Contin. Dinam. Systems, 7 (2001), pp.~247--258.

\bibitem{BDFerr}
{\sc G.~Buttazzo, L.~De~Pascale, and I.~Fragal\`{a}}, {\em Erratum}, Discrete
  Contin. Dinam. Systems, 18 (2007), pp.~219--220.

\bibitem{dalmaso} {\sc G. Dal Maso}, \emph{An introduction to $\Gamma$-convergence}, Springer Science+Business Media, New York, 1993.

\bibitem{Smooth}
{\sc  A.~Davini},
{\em Smooth approximation of weak {F}insler metrics}, Differential and Integral Equations, 18, 5 (2005), pp.~509--530.

\bibitem{DCP2}
{\sc G.~De~Cecco and G.~Palmieri},
{\em L{IP} manifolds: from metric to {F}inslerian structure}, Math Z., 207 (1991), pp.~223--243.

\bibitem{esspasq}
{\sc F.~Essebei and E.~Pasqualetto},
{\em Variational problems concerning sub-Finsler metrics in Carnot groups}, ESAIM: Control, Optimisation and Calculus of Variations, 29 (2023), doi:10.1051/cocv/2023006.

\bibitem{Spivak}
{\sc M.~Spivak}, \emph{A Comprehensive Introduction to Differential Geometry, volume 1},
Publish or Perish, Incorporated, 1999.

\bibitem{Stone}
{\sc A.~H.~Stone},
{\em Paracompactness and product spaces}, Bulletin of the American Mathematical Society, 54  (1948), pp.~977--982.

\end{thebibliography}
\end{document}